\documentclass[12pt,pdftex,noinfoline]{imsart}
\usepackage[usenames,dvipsnames]{xcolor}
\usepackage{natbib}
\bibpunct{(}{)}{;}{a}{,}{,}
\usepackage{verbatim}
\usepackage{comment}
\usepackage[pdftex,pdfborderstyle={/S/U/W 1},colorlinks]{hyperref}
\RequirePackage[OT1]{fontenc}
\usepackage[linesnumbered]{algorithm2e}
\RequirePackage{amsthm,amsmath,amsfonts}
\RequirePackage{hypernat}
\usepackage{fullpage}
\usepackage{graphicx}
\usepackage{hyperref}
\usepackage{amsthm, amssymb, amsmath}
\usepackage{url}
\usepackage[english]{babel}
\usepackage{times}
\usepackage[T1]{fontenc}
\usepackage{bbm}
\usepackage{xspace}
\usepackage{rotating,multirow}
\usepackage{enumerate}
\usepackage{mathtools}
\usepackage{tikz}

\def\interleave{|\kern-.25ex|\kern-.25ex|}
\def\interleavesub{|\kern-.15ex|\kern-.15ex|}

\newcommand{\nNorm}[1]{\left|\kern-.25ex\left|\kern-.25ex\left| {#1}\right|\kern-.25ex\right|\kern-.25ex\right|}

\newcommand{\E}{{\mathbb E}}

\newcommand{\eps}{{\epsilon}}

\newcommand{\R}{{\mathbb R}}

\renewcommand{\P}{{\mathbb P}}

\newcommand{\A}{{\mathcal{A}}}
\newcommand{\M}{{\mathcal{M}}}

\newcommand{\I}{{\mathcal{I}}}

\newcommand{\T}{{\mathcal{T}}}

\newcommand{\G}{{\mathcal{G}}}

\newcommand{\F}{{\cal F}}

\def\min{\mathop{\text{\rm min}}}
\def\max{\mathop{\text{\rm max}}}
\def\inf{\mathop{\text{\rm inf}}}
\def\sup{\mathop{\text{\rm sup}}}

\numberwithin{equation}{section}
\theoremstyle{plain}
\newtheorem{theorem}{Theorem}[section]
\newtheorem{corollary}{Corollary}[section]
\newtheorem{proposition}{Proposition}[section]
\newtheorem{lemma}{Lemma}[section]
\newtheoremstyle{remark}{\topsep}{\topsep}%
     {\normalfont}
     {}           
     {\bfseries}  
     {.}          
     {.5em}       
     {\thmname{#1}\thmnumber{ #2}\thmnote{ #3}}
\theoremstyle{remark}
\newtheorem{remark}{Remark}[section]

\newtheorem{definition}{Definition}[section]

\long\def\comment#1{}

\def\P{{\mathbb P}}
\def\E{{\mathbb E}}
\def\supp{\mathop{\text{supp}\kern.2ex}}

\def\argmax{\mathop{\text{\rm arg\,max}}}
\let\hat\widehat
\let\tilde\widetilde

\let\hat\widehat
\let\tilde\widetilde

\def\1{{(1)}}
\def\2{{(2)}}

\def\cG{{\mathcal{G}}}
\def\M{{\mathcal{M}}}

\long\def\comment#1{}

\long\def\comment#1{}

\def\P{{\mathbb P}}
\def\E{{\mathbb E}}

\def\supp{\mathop{\text{supp}\kern.2ex}}

\def\argmax{\mathop{\text{\rm arg\,max}}}
\let\tilde\widetilde

\let\hat\widehat
\let\tilde\widetilde

\def\1{{(1)}}
\def\2{{(2)}}

\def\cG{{\mathcal{G}}}

\long\def\comment#1{}

\def\threebars{\mbox{$|\kern-.25ex|\kern-.25ex|$}}

\def\H{\mathcal{H}}

\def\F{\mathbb{F}}

\begin{document}

\hypersetup{citecolor=MidnightBlue}
\hypersetup{linkcolor=Black}
\hypersetup{urlcolor=MidnightBlue}

\begin{frontmatter}

\mbox{}
\vskip.25in
\title{Estimation in Tournaments and Graphs under Monotonicity Constraints}
\maketitle

\begin{aug}
\author{\fnms{Sabyasachi} \snm{Chatterjee}\ead[label=e1]{sabyasachi@galton.uchicago.edu}} \and
\author{\fnms{Sumit} \snm{Mukherjee}\ead[label=e2]{sm3949@columbia.edu}} 



\runauthor{Chatterjee, S., Mukherjee, S.}

\affiliation{University of Chicago and Columbia University}

\address{5734 S.~University Avenue \\
Chicago, IL 60637 \\
\printead{e1}
}

\address{1255 Amsterdam Avenue \\
New York, NY 10027\\
\printead{e2}
}
\end{aug}

\begin{abstract}
We consider the problem of estimating the probability matrix governing a tournament or linkage in graphs from incomplete observations, under the assumption  that the probability matrix satisfies natural monotonicity constraints after being permuted in both rows and columns by some latent permutation. 
We propose a natural estimator which bypasses the need to search over all possible latent permutations and hence is computationally tractable. We then derive asymptotic risk bounds for our estimator. Pertinently, we demonstrate an automatic adaptation property of our estimator for several sub classes of our parameter space which are of natural interest, including  generalizations of the popular Bradley Terry Model in the Tournament case, the $\beta$ model and Stochastic Block Model in the Graph case, and H\"older continuous matrices  in the tournament and graph settings. 
\end{abstract}


\vskip15pt
\end{frontmatter}

\maketitle
\section{Introduction}
In this paper we consider two statistical estimation problems. We begin by describing the two set ups.
\begin{itemize}
\item Consider the situation of $n$ teams playing in a league tournament where each team plays every other team once. The results of the tournament can be written as a data matrix $y$ of zeroes and ones by setting $y_{ij} = 1$ for $i<j$ if team $i$ wins against team $j$, and $0$ otherwise.  Let $\theta_{ij}$ be the probability that team $i$ wins against team $j$ with $\theta_{ji} = 1 - \theta_{ij}$ whenever $i \neq j.$ Set $\theta_{ii} = 0$ for all $1 \leq i \leq n$ as a matter of convention. The upper triangular part of the data matrix $y$ is modeled as 
\begin{equation}
y_{ij} \sim Bern(\theta_{ij}), \:\:\forall 1 \leq i \le j \leq n
\end{equation}
where $y_{ij}$ in the upper triangular part is jointly independent and $Bern(.)$ refers to the standard Bernoulli distribution.

The lower triangular part of the data matrix is filled in a skewsymmetric manner; that is
\begin{equation}
y_{ij} = 1 - y_{ji},   \:\:\forall 1 \leq j < i \leq n.
\end{equation}
The problem then is to estimate the pairwise comparison probability matrix $\theta$ based on the observed data matrix $y.$ This setting can arise whenever the data is in the form of pairwise comparisons (see~\cite{david1963method}), for example in analyzing customer preferences for items, citation patterns for journals (see~\cite{stigler1994citation}). For convenience, we stick to the tournament terminology in this paper. Note that we have $O(n^2)$ parameters to estimate and $O(n^2)$ data points. Hence one needs structural assumptions on $\theta$ for consistent estimation to be possible. The classical approach in this problem is to assume that the pairwise probability matrix has the following structural form:
\begin{equation}\label{bradleyterry}
\theta_{ij} = \frac{\exp(w_i)}{\exp(w_i) + \exp(w_j)}.
\end{equation}
where the $w = (w_1,\dots,w_n)$ vector is a vector of weights representing the skill/ability of the teams. This is the Bradley-Terry model (see~\cite{bradley1952rank}) which is very popular in the ranking literature. It is then common to estimate $w$ by maximum likelihood (see~\cite{hunter2004mm}) and plug it in to estimate $\theta_{ij}.$ The study of asymptotic estimation of $w$ in the Bradley Terry Model  has a long history in Statistics (see~\cite{simons1999asymptotics} and references therein).

We are interested in the problem of estimating the matrix of probabilities $\theta_{ij}$ under an assumption commonly made in the ranking literature known as Strong Stochastic Transitivity (\textbf{SST}) (see~\cite{shah2016stochastically} and references therein). This assumption posits the existence of an ordering among the teams which is unknown to the statistician. This ordering is then reflected on the probabilities $\theta_{ij}$ as follows. Let team $j$ have a higher rank than team $k$  (i.e. team $j$ is better than team $k$). Then for any team $i$, the probability of team $i$ defeating team $k$ would be no less than the probability of team $i$ defeating team $j$, which gives $\theta_{ij} \leq \theta_{ik}.$

Even though the SST condition is classical, a formal study of estimation under this condition was done recently in \cite{ChaUSVT15} and was termed as the \textit{Nonparametric Bradley Terry Model}. The terminology is apt because it clearly generalizes the very commonly used Bradley Terry model. Any matrix $\theta$ of the Bradley Terry form in~\eqref{bradleyterry} satisfies the SST condition with the ordering given by the ordering of the $w$ vector. We refer to Proposition $1$ in~\cite{shah2016stochastically} who show in a precise sense that the Nonparametric Bradley Terry Model is a significant extension of the usual Bradley Terry model.

In many realistic scenarios, we would not be able to observe all pairwise comparisons. Thus, it is of interest whether one can still estimate the pairwise comparison matrix $\theta$ in the situation where we observe only a fraction of all possible games that can be observed. In this paper we consider the missing data at random setting. In this setting we get to observe each entry above the diagonal with probability $p$ independently of other observations above the diagonal.

The purpose of this paper is to propose and analyze a computationally tractable estimator in this problem. The main focus of this paper is to obtain finite sample risk bounds (upto a constant factor) and study how small can $p$ be to still allow consistent estimation. 
\\

\item Consider now the situation of observing a random graph on $n$ nodes with no self loops. Let $\theta_{ij}$ now be the probability of node $i$ and node $j$ being linked. Again we set $\theta_{ii} = 0$ for all $1 \leq i \leq n$ as a matter of convention. The random graph can be now encoded as an adjacency matrix $y$ of zeroes and ones. Again, the upper triangular part of the adjacency matrix is modelled as
\begin{equation}\label{graph}
y_{ij} \sim Bern(\theta_{ij})\:\:\forall\:\:1 \leq i \leq j \leq n
\end{equation}
where $y_{ij}$ in the upper triangular part is jointly independent. The lower triangular part of the data matrix is now filled in a symmetric manner; that is
\begin{equation}
y_{ij} = y_{ji}\:\:\forall\:\:1 \leq j < i \leq n.
\end{equation}
Inspired by the SST assumption in the ranking literature, here we assume that that the vertices can be arranged in an order (unknown to the statistician) of increasing tendency of getting linked to other vertices. This assumption will again impose monotonicity constraints on the edge probabilities $\theta_{ij}.$ For example if node $j$ is more "active" or "popular" than node $k$ then for any node $i$ we must have $\theta_{ik} \leq \theta_{ij}.$ For an example where such an assumption seems natural, consider a social network with $n$ people labeled $\{1,2,\cdots,n\}$ where the $i^{th}$ person has a popularity parameter $p_i\in [0,1]$. The chance that person $i$ and person $j$ are friends is  $f(p_i,p_j)$, where $f$ is increasing in both co-ordinates to signify that increasing popularity leads to more friendship ties. The function $f$ also needs to be symmetric, as the chance that $i$ and $j$ are friends is symmetric in $(i,j)$. Indeed, in this case there is (at least) one ordering which sorts the nodes of the network in increasing order of popularity.

We pose and study the problem of estimating the edge probability matrix $\theta$ in this set up  with missing data at random. We sometimes refer to this model of random graphs as the symmetric model, differentiating it from the skewsymmetric (tournament) case. Under our model assumptions, the problem of estimating the edge probabilities is very closely related to the problem of estimating graphons in the spirit of~\cite{gao2015rate} where we assume monotonicity (without smoothness) of the graphon in both variables, instead of smoothness assumptions made in~\cite{gao2015rate}. 
\end{itemize}

In this paper we look at the above two estimation problems in the skewsymmetric and the symmetric model in a  unified way. In particular, we introduce and study the risk properties of a natural estimator which is described in subsection~\ref{algo}. Our estimator has the same form in both the models, and the technique of analyzing the risk properties of the estimator in both the models is the same.

\subsection{Formal Setup of our problem}\label{formal}
In this subsection we define two parameter spaces; one for the skewsymmetric model and one for the symmetric model. 

Denote $S_n$ to be the set of all permutations on $n$ symbols. For any $n \times n$ matrix $\theta$ and any permutation $\pi \in S_n$ we define $\theta \circ \pi$ to be the $n \times n$ matrix such that $(\theta \circ \pi)_{ij} = \theta_{\pi(i),\pi(j)}.$ If $\Pi$ denotes the $n \times n$ permutation matrix corresponding to the permutation $\pi \in S_n,$ then we have  $\theta \circ \pi = \Pi^T \theta \:\Pi,$ where the RHS is usual matrix multiplication.

Let $\mathcal{T}$ be the space of tournament matrices defined by
\begin{align}
\T :=\{\theta\in [0,1]^{n\times n}:\:\:&\theta_{ij} \leq \theta_{ik}\ \:\text{whenever}\:  i \neq k, i \neq j, k < j ;\nonumber\\&\theta_{ji}=1-\theta_{ij}\:\:\forall\:\:i \neq j; \theta_{ii} = 0\:\ \forall i;\:\theta_{ij} \leq \frac{1}{2} \:\forall i < j\}.
\end{align}
Any matrix in $\T$ when only looked at the upper triangular part above the diagonal is non increasing in any row (as $j$ grows) and non decreasing in any column (as $i$ grows). The lower triangular part is just $1$ minus the upper triangular part and the diagonals are zero. In words, $\T$ is the space of matrices which satisfy the SST assumption with known ranking where the ranking is such that player $n$ is the best, followed by player $n - 1$ and so on. Then our parameter space for the skewsymmetric model can be written as
\begin{equation}
\Theta_{\T} = \{\theta \circ \pi: \theta \in \T, \pi \in S_n\}.
\end{equation}

Similarly, define the space of matrices 
\begin{align}
\G =\{\theta\in \R^{n\times n}:\theta_{ij} \le \theta_{ik}\ \text{whenever}\:  i \neq k, i \neq j, j < k ;\quad \theta_{ji}=\theta_{ij}\:\:\forall i \ne j;\quad \theta_{ii}=0\ \forall i\}.
\end{align}
Any matrix in $\G$ when only looked at the upper triangular part above the diagonal is non decreasing in both rows and columns. The lower triangular part is symmetrically filled, and the diagonals are zero. Again, $\G$ is the space of expected adjacency matrices which are consistent with the monotonicity restrictions imposed by the ordering where node $n$ is most popular followed by node $n - 1$ and so on. Then our parameter space for the symmetric model can be written as
\begin{equation}
\Theta_{\G} = \{\theta \circ \pi: \theta \in {\G}, \pi \in S_n\}.
\end{equation}

For $\Theta=\Theta_{\T}$ or $\Theta_{\G}$ we study the problem of estimating the underlying matrix of probabilities $\theta$. The loss function we consider is the mean Frobenius squared metric defined for any two matrices $\theta$ and $\tilde{\theta}$ as 
$
\frac{1}{n^2}\|\tilde{\theta}-\theta\|^2$ , where $\|A\|$ denotes the Frobenius norm of the matrix $A.$

\subsection{The estimator}\label{algo}
The purpose of this paper is a statistical study of the following estimator, denoted by $\hat{\theta},$ of the true underlying mean matrix $\theta^* \circ \pi^*$ where $\theta^* \in \mathcal{T}/\cG$, 
and $\pi^* \in S_n$ where $S_n$ is the space of permutations on $n$ symbols. Construction of our estimator consists of several steps. Initially our data matrix $y$ will have some entries $1$ or $0$ and some entries would be missing. Let $\hat{p}$ denote the proportion of non missing entries in our data, i.e. $\hat{p}$ equals the number of observed games divided by ${n \choose 2}.$
\begin{enumerate}
\item[(a)]
{\bf Filling}

Fill the missing entries of the data matrix $y$ by $1/2$. Henceforth let us denote $y$ to be the data matrix obtained after filling up the missing entries.

\item[(b)]
{\bf Estimate Ranking}

Let $r_i = \sum_{j = 1}^{n} y_{ij}$ be the $i^{th}$ row sum of the data matrix $y.$ In this step we sort the vertices according to the row sums $(r_1,\dots,r_n)$ of the data matrix $y$ and obtain a permutation $\hat{\sigma}$ such that $r_{\hat{\sigma}(1)} \leq \dots \leq r_{\hat{\sigma}(n)}.$ The inverse of the random permutation $\hat{\sigma}$ can be thought of as a proxy for the underlying $\pi^{*}.$  
In case there are ties in the vector $r = (r_1,\dots,r_n),$ break the ties uniformly at random while obtaining the sorting permutation $\hat{\sigma.}$


\item[(c)]
{\bf Debiasing}

In this step we transform $y$ to $\big(\frac{y - J/2}{\hat{p}} + J/2\big) $ where $J$ is a matrix with $J_{ij} = 1$ for any $1 \leq i \neq j \leq n$ and $J_{ii} = 0$ for any $1 \leq i \leq n.$ 
  
\item[(d)]
{\bf Sorting and Projecting}

We then sort the debiased data matrix by applying to it the sorting permutation $\hat{\sigma}$ obtained from Step $2.$ We then project $\big(\frac{y - J/2}{\hat{p}} + J/2\big) \circ \hat{\sigma}$ onto the relevant parameter space. In the skewsymmetric model, we project onto the set $\T \subset [0,1]^{n \times n}$ and in the symmetric model, we project onto the set $\G \subset [0,1]^{n \times n}.$ Both $\T$ and $\G$ are closed convex sets of matrices and hence there exists a unique projection onto them. Let the projection operator be denoted by $Proj$ in both cases. After this step we have the projection of a sorted and debiased data matrix $Proj\Big(\big(\frac{y - J/2}{\hat{p}} + J/2\big) \circ \hat{\sigma}\Big).$

\item[(d)]
{\bf Unsorting}

We now unsort the debiased, sorted and projected data matrix  by applying to it the inverse of the sorting permutation $\hat{\sigma}^{-1}$.  
\item[(e)]
{\bf Estimate by $1/2$ if too many missing entries}

We now define our final estimator $\hat{\theta}$ as follows:
\begin{align}\label{estimator}
\hat{\theta}:= &Proj\Big(\big(\frac{y - J/2}{\hat{p}} + J/2\big) \circ \hat{\sigma}\Big) \circ \hat{\sigma}^{-1} \:\:\text{ if }\hat{p}\ge \frac{1}{n},\\
\notag=&\frac{J}{2}\text{ otherwise}.
\end{align}
\end{enumerate}

\begin{remark}
The transformation $y \rightarrow \frac{y - J/2}{\hat{p}} + J/2$ is merely a debiasing step in the following sense. Let $\pi^*$ be the identity permutation for simplicity. For any fixed $1 \leq i \neq j \leq n,$ the random variable $y_{ij}$ thus takes the value $1$ with probability $p \theta^*_{ij},$ takes the value $0$ with probability $(1 - p) \theta^*_{ij}$ and takes the value $1/2$ with probability $(1 - p).$ Thus $E\:\:\frac{y_{ij} - 1/2}{p} + 1/2 = \theta^*_{ij}.$
\end{remark}
\begin{remark}
Note that in the skewsymmetric (tournament) model, the row sums of the data matrix $y$ correspond to the number of wins or victories for each player, and the column sums of the data matrix correspond to the number of defeats for each player. Hence our sorting step just sorts the teams according to the number of victories (or equivalently the number of defeats, as sum of victory and defeat of each player is $n-1$). Similarly, in the symmetric (graph) model, the row sums of the adjacency matrix $y$ correspond to the empirical degrees of the nodes in the graph. Therefore, our sorting step sorts the vertices according to the empirical degrees.
\end{remark} 

The first step of our algorithm just needs computation of the row sums and then sorting, which combined clearly will only take at most $O(n^2)$ operations. The projection step is thus going to be dominating the computation time. Fortunately there are efficient ways of computing the projection. First of all, the projection by definition has to be zero on the diagonals and would be skew symmetric/symmetric according to our model. Hence in either of the models, it suffices to compute the projection for the upper diagonal part. It turns out that the spaces $\T$ and $\G,$ without the constraint that all elements have to live in $[0,1],$ can be viewed as the space of Isotonic functions on an appropriate Directed Acyclic Graph (DAG) on the domain $\{(i,j): 1 \leq i < j \leq n\}.$ Isotonic functions on a DAG can only increase on following a directed path in the DAG.  Recent results in~\cite{kyng2015fast} study how to compute such Isotonic projections on a general DAG. Their result imply an algorithm with runtime $O(n^3)$ for our problem.  
More classically, the problem of computing this projection is closely related to computing what is called a Bivariate Isotonic Regression(see~\cite{chatterjee2015matrix}). There exists efficient iterative algorithms for this purpose; see page 27 in~\cite{RWD88}. The focus of this paper therefore is not on computation of our estimator but rather on its statistical properties.

\section{Main result}\label{results}
Having defined our estimator, we now make a couple of more definitions. 
\begin{definition}
Henceforth we will use the notation $\Theta$ in place of $\Theta_\mathcal{T}$ or $\Theta_\mathcal{G}$. The implication is that all results hold with $\Theta$ replaced by either of the two parameter sets. 

Setting $[n]:=\{1,2,\cdots,n\}$,  for any $\theta\in \Theta$ define the row sums $R_i(\theta):=\sum_{j\in [n]}\theta_{ij}$ for $i\in [n]$. Also for each $\theta \in \Theta$ and $p\in [0,1]$ define the quantity  $$Q(p,\theta):=\sum_{i\in [n]}\max_{j:|R_j(p\:\theta)-R_i(p\:\theta)| \leq 4\sqrt{np\log n}} \sum_{k = 1}^{n} (p\:\theta_{ik}- p\:\theta_{jk})^2.$$
\end{definition}
The quantity $Q(p,\theta)$ is an important quantity in our analysis. As will be clear from our subsequent analysis, matrices in $\Theta$ with additional structure tend to have smaller $Q(p,\theta).$ We are now ready to state the main theorem of this paper. 
\begin{theorem}\label{main}
Consider the estimator $\hat{\theta}$ defined in~\eqref{estimator}. Let $n \geq 2$ and let $p \geq \frac{2\log n}{n}.$ Then there exists a universal constant $C < \infty$ such that for any $\theta^* \in \Theta$ we have   
\begin{align}
\frac{1}{n^2}\E\|\hat{\theta} - \theta^*\|^2\leq C \Big[\frac{(\log n)^2}{np} + \frac{Q(p,\theta^*)}{n^2 p^2}\Big].
\end{align}
\end{theorem}



\begin{remark}
Theorem~\ref{main} gives an upper bound which is a sum of two terms. The first term scaling like $\tilde{O}(1/(np))$ can be thought of as the risk arising due to the shape constraint imposed by the SST condition. The second term involving $Q(p,\theta^*)$ can be interpreted as the risk arises because of our sorting step; it measures how much $\theta^*$ changes in a frobenius norm squared sense, if its rows/columns are permuted by a typical sorting permutation. 
\end{remark}

\begin{remark}
The strength of Theorem \ref{main} is that it is adaptive in the parameter $\theta^*$, and gives tight asymptotic bounds for several sub-parameter spaces of interest. The term $ \frac{Q(p,\theta^*)}{n^2 p^2}$ dominates the $\frac{(\log n)^2}{np}$ term in most cases of interest. Hence the quantity $Q(p,\theta^*)$ determines the rate of convergence of our estimator.

All our results will be derived as corollaries of our main theorem. We now begin to discuss the various consequences and implications of Theorem~\ref{main}.
\end{remark}
\subsubsection{\textbf{Worst case Risk Bound}}
As a first application of Theorem~\ref{main}, we deduce the worst case risk of our estimator.
\begin{corollary}\label{cor:worst}
There is a universal constant $C<\infty$ such that for any $p\ge \frac{2\log n}{n}$ and $n \geq 2$ we have
\begin{align}
\sup_{\theta^* \in \Theta} \frac{1}{n^2}\E\|\hat{\theta} - \theta^*\|^2\leq C \Big[\frac{(\log n)^2}{np}+\sqrt{\frac{\log n}{np}}\Big].
\end{align}

\end{corollary}

\begin{remark}
The brute force LSE in our problem achieves a MSE of $O(1/(np))$ upto log factors, which is minimax rate optimal (upto log factors) in the tournament setting (see Theorem 5(a) in~\cite{shah2016stochastically}). The same fact can be shown to be true in the graph setting by an application of Lemma 3.1 in~\cite{chatterjee2015adaptive}. We do not carry this out in this manuscript.

Comparing Theorem~\ref{main} to the minimax rate, the $\tilde{O}((np)^{-\frac{1}{2}})$ rate achieved by our estimator is clearly worse than the minimax rate of estimation. However the only method known to achieve the minimax rate is the LSE, which is perhaps not computationally feasible. This raises the important question of whether there exists a computationally feasible estimator achieving the minimax rate in this problem. This question deserves further study and is beyond the scope of the current manuscript.

The authors in~\cite{shah2016stochastically} (see Theorem $5b$) improved the analysis of an estimator based on singular value threshholding proposed in~\cite{ChaUSVT15} and demonstrated its rate of convergence to be $\tilde{O}((np)^{-\frac{1}{2}}).$ Hence our estimator matches the best known rate of convergence for computationally feasible estimators.  
\end{remark}

\begin{remark}
It follows from Corollary \ref{cor:worst} that the MSE converges to $0$ as soon as $np\gg (\log n)^2$. 
On the other hand, if $np$ converges to a finite number $\lambda$, the known minimax lower bound implies that the MSE stays bounded away from $0$. 
\end{remark}


A natural question is whether the upper bound for our estimator in Corollary~\ref{cor:worst} is tight. The following example suggests that the given upper bound is tight at least when $p = 1$.
\\
Define $\theta^{wc} \in \T$ for any $i \neq j$ as 
\begin{align}\label{hard}
\theta^{wc}_{i,j} =& \frac{1}{2} + \frac{1}{4} \I\{i > j\}.
\end{align}

Simulations of the MSE (see figure \ref{fig:2}) suggest that our exponent of $n$ in the upper bound in Corollary~\ref{cor:worst} is tight.
\begin{figure*}[ht]
\begin{center}
\includegraphics[scale=.6]{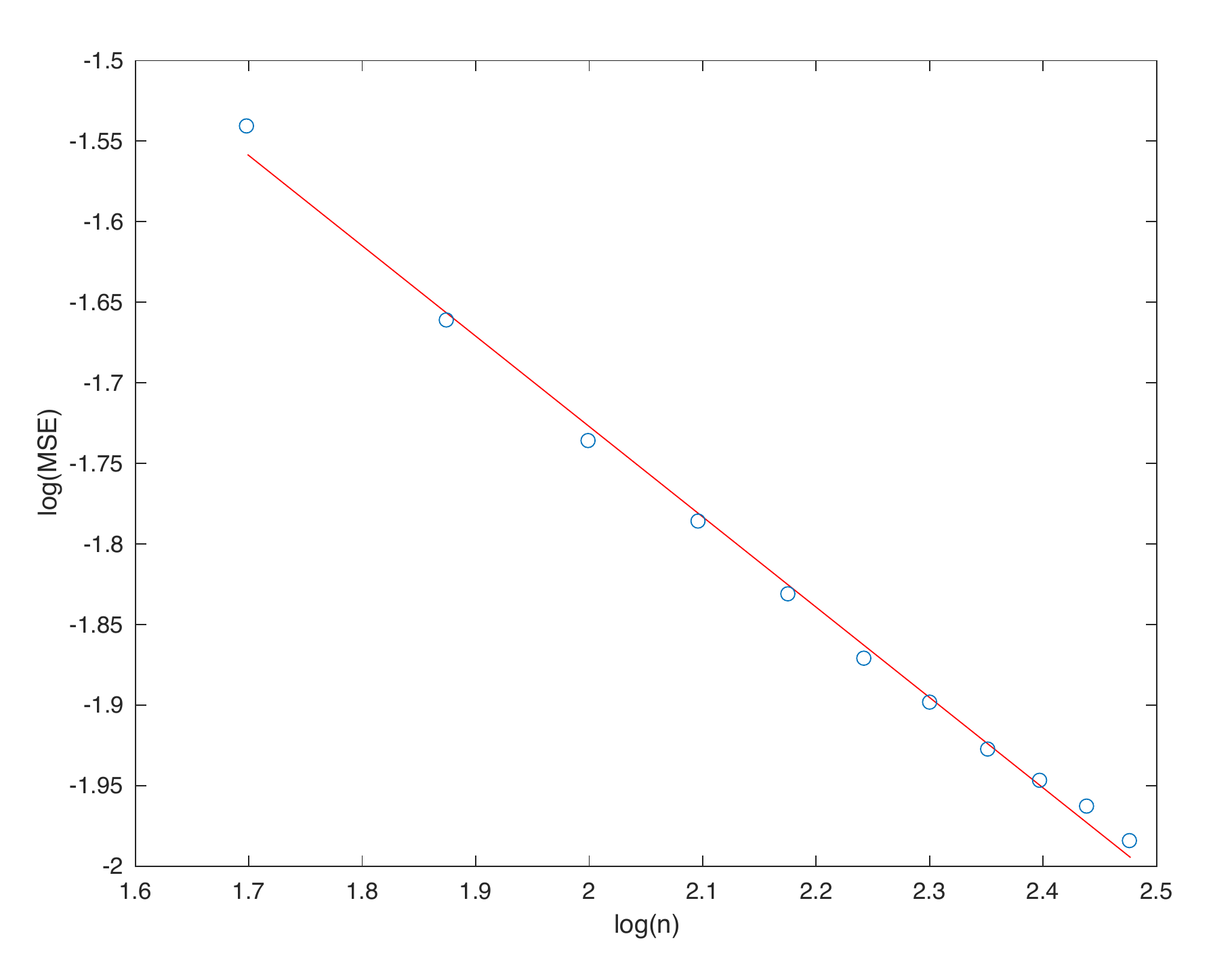}
\caption{\text{MSE for $\theta=\theta^{wc}$ (worst case behavior)  }}
\end{center}
\label{fig:2}
\end{figure*}
For the plot we have simulated our data when the true underlying $n \times n$ parameter matrix is given by~\eqref{hard}. We have simulated $100$ times for each sample size $n$ ranging from $50$ to $300$ in increments of $25.$ 
 A plot of log(n) (base $10$) versus log(MSE) clearly shows a linear behavior, and the best fitted line has slope $-0.5607$, which is close to the predicted exponent $-0.5.$ 



\subsection{Adaptive Risk Bounds}
The main reason for studying our estimator is that it exhibits automatic adaptation properties. By automatic adaptation we mean that even though our estimator achieves $\tilde{O}(n^{-\frac{1}{2}})$ rate of estimation globally, it achieves provably faster rates of estimation for several subclasses of our parameter space which are of independent interest. We now describe some of these subclasses along with the associated risk bounds. 

\subsubsection{Block matrices}
Let us consider the class of matrices within our parameter space $\Theta$ which are a permuted version of some $k \times k$ blockmatrix. This class of matrices (without the restriction of being within $\Theta$) is popularly known as the Stochastic Blockmodel in Network Analysis. In the Tournament setting, one can think of the situation where there are $n$ teams but only $k \ll n$ levels of teams. Within each level, each team has the same ability. Therefore, assuming a block structure on the true pairwise comparison matrix is a natural way to impose sparsity in the model. 
\begin{definition}
For $k\in [n]$, let $\Theta^{(k)}\subseteq \Theta$ denote the subset of all $k\times k$ block matrices with equal sized blocks, upto an unknown permutation. 
\end{definition}
\begin{corollary}\label{cor:block}
There exists a universal constant $C<\infty$ such that for any $p\ge \frac{2\log n}{n}$ we have
$$\sup_{\theta^* \in \Theta^{(k)}}\frac{1}{n^2}\E\|\hat{\theta}-\theta^*\|^2\le C\Big[\frac{\min(k,\sqrt{np})\log n}{np}+\frac{(\log n)^2}{np}\Big].$$

\end{corollary}

%


%

\begin{remark}
In particular when $k$ is fixed, we get from corollary \ref{cor:block} that
$$\frac{1}{n^2} \E |\hat{\theta}-\theta^*\|^2\le C \frac{(\log n)^2}{np}.$$
The next theorem shows that the above upper bound is unimprovable in the sense that this is the minimax error rate for $2\times 2$ blocks of equal size, upto $\log$ factors. 
In the case when the latent ranking is known, the estimation problem is essentially equivalent to Bivariate Isotonic Regression (see~\cite{chatterjee2015matrix}), where the minimax error rate for $2 \times 2$ block matrices (without a latent permutation) is shown to be $\tilde{O}(n^{-2}).$  This means that the latent permutation in our parameter space makes the statistical estimation problem fundamentally harder in a minimax sense. 
\end{remark}

\begin{theorem}\label{minimax1}
Consider the subset $\Theta^{(2)} \subset \Theta$ of all equal sized $2 \times 2$ block matrices. Then for some universal constant $c>0$ we have
\begin{equation*}
\inf_{\tilde{\theta}} \sup_{\theta \in \Theta^{(2)}} \frac{1}{n^2} \sum_{i = 1}^{n} \sum_{j = 1}^{n} \E (\tilde{\theta}_{ij} - \theta_{ij})^2 \geq \frac{c}{np}.
\end{equation*}
\end{theorem}

\textit{The above theorem shows that when $k$ is fixed our estimator is minimax optimal upto logarithmic factors for the class $\Theta^{(k)}.$}

\begin{remark}
It is instructive to compare with known results when $p = 1.$ Without the monotonicity constraint, it was shown in \cite{gao2015rate} that for symmetric $k\times k$ block matrices the minimax rate is $O(\frac{k^2}{n^2}+\frac{\log k}{n})$. Thus for $k\in [1,\sqrt{n}]$ the minimax rate with and without monotonicity constraints are both $\tilde{O}(n^{-1})$. For $k\in [\sqrt{n},n]$ the minimax rate with monotonicity is $\tilde{O}(n^{-1})$, whereas without monotonicity the minimax rate is strictly larger. 
\end{remark}

\subsubsection{Generalized Bradley Terry model + Generalized $\beta$ model}
In the Tournament setting, if the Bradley Terry model is correct, one can use the MLE of the weights vector $w$ to estimate $\theta.$ It is known that this estimator attains a fast rate of convergence of $O((np)^{-1})$ (see Theorem $5c$ in~\cite{shah2016stochastically}). A natural question is whether our estimator attains this fast rate of convergence whenever the true matrix is actually of the Bradley Terry form. Our next corollary shows this to be true in a more general sense. Before stating the corollary, let us define the Generalized Bradley Terry Model $\Theta_{gbt,M}$ contained within the class of SST matrices. 
\begin{definition}
Let $\Theta_{gbt,M} \subset \Theta_\T$ denote the subset of all tournament matrices such that $\theta_{i,j}=F(w_i-w_j)$, where $F$ is a unknown symmetric distribution function on $\R$ (i.e. $F(x)+F(-x)\equiv 1$) with a continuous strictly positive density function function, and $\{w_i\}_{i\in [n]}$ is a sequence of real numbers lying in a compact interval $[-M,M]$. It is not hard to check that the class $\Theta_{gbt,M}\subset \Theta_\T$ indeed satisfies the SST condition. In particular, setting $F(x)=\frac{e^x}{1+e^x}$ and $F(x)=\Phi(x)$ (the normal distribution function) we get the usual Bradley Terry model and the Thurstone model (see~\cite{bradley1952rank,thurstone1927law}) respectively. 
Let us call this class of matrices the Generalized Bradley Terry Model. 
\end{definition}

Let us now define the analogous class of matrices in the symmetric model on graphs. 

\begin{definition}
Let $\Theta_{gbm,M}\subset \Theta_\cG$ denote the subset of all expected adjacency matrices such that $\theta_{i,j}=F(w_i+w_j)$, where $F$ is  a  distribution function on $\R$ with a continuous strictly positive density function, and $\{w_i\}_{i\in [n]}$ is a sequence of real numbers in  $[-M,M]$.  In particular for the choice $F(x)=\frac{e^x}{1+e^x}$ gives the $\beta$ model on networks, which has originated from Social Sciences and has been studied in Statistics(c.f.~\cite{ChatDiaSly},~\cite{mukherjee2016detection} and references there-in). The class $\Theta_{gbm,M}$ generalizes the usual $\beta$ model to allow for a more general class of distribution functions.
\end{definition}

%
%
%

With $F$ assumed to be known, estimation in the Generalized Bradley Terry model has been studied by~\cite{shah2016stochastically}. We present our next corollary treating $F$ as an unknown parameter. 
\begin{corollary}\label{cor:gbm}
There exists a universal constant $C < \infty$ such that for any $p \geq \frac{2\log n}{n}$ we have
$$\sup_{\theta^* \in \Theta_{gbt,M}} \frac{1}{n^2} \E\|\hat{\theta}-\theta^*\|^2\le C\Big[\Big(\frac{L}{L'}\Big)^2\frac{\log n}{np}+\frac{(\log n)^2}{np}\Big],$$
where $L:=\sup_{|x|\le 2M}f(x)$ and $L':=\inf_{|x|\le 2M}f(x)$.
The same risk bound holds for the Generalized Beta Model $\Theta_{gbm,M}.$
\end{corollary}


If the distribution function $F$ is known, the authors in~\cite{shah2016stochastically} (Theorem $5c$) show that the estimator based on the MLE of $w$ achieves the minimax optimal rate $O((np)^{-1})$, under the stronger condition that the density $f$ is strongly log concave and twice differentiable. \textit{This implies that our estimator is minimax rate optimal upto log factors for the bigger class  $\Theta_{gbt,M}$ without using the knowledge of $F$.} Obviously, computing the MLE takes into account the explicit knowledge of $F.$ The estimator studied in this paper attains the same fast rate of convergence as the MLE upto log factors, but without the knowledge of $F.$ 


\subsubsection{Smooth matrices}
In the symmetric model (graph setting) there has been a tradition of studying estimation of graphons satisfying some smoothness condition (see~\cite{gao2015rate} and references therein). We think it is a natural question to ask what happens when the true pairwise probability matrix in the Tournament setting (in addition to satisfying SST) is a smooth matrix, upto an unknown permutation. This motivates us to define the following class of matrices satisfying a discrete version of the usual Holder smoothness conditions. 
\begin{definition}
For $\alpha,L>0,$ let $\Theta(\alpha,L)\subset \Theta$ denote the subset of all permuted versions of Holder continuous matrices with order $\alpha$ and Holder constant $L$ , i.e. $\theta,$ after being permuted in rows and columns, satisfies $$|\theta_{i,j}-\theta_{i,k}|\le L\frac{|j-k|^\alpha}{n^\alpha}$$ for all $i \neq j,i \neq k,i,j,k\in [n]$. 
\end{definition}


The next corollary shows that our estimator provably attains faster rates of convergence than $\tilde{O}(n^{-\frac{1}{2}})$ whenever the true matrix, upto an unknown permutation, satisfies smoothness conditions as above. 
\begin{corollary}\label{uppersmooth}
There exists a universal constant $C<\infty$ such that for any $p\ge \frac{2\log n}{n}$ we have
$$\sup_{\theta^* \in \Theta(\alpha,L)} \frac{1}{n^2} \E\|\hat{\theta}-\theta^*\|^2\le C\Big[L^{\frac{1}{\alpha+1}}\frac{\log n}{(np)^{\frac{2\alpha+1}{2\alpha+2}}}+\frac{(\log n)^2}{np}\Big]$$

\end{corollary}

\begin{remark}
As the above corollary shows, the adaptation of our estimator depends crucially on the order of the Holder class $\alpha$. In particular, for the class of Lipschitz matrices (which correspond to the choice $\alpha=1$) with no missing entries (which correspond to $p=1$), the corollary implies the following upper bound:
\begin{align*}
\sup_{\theta^* \in \Theta(1,1)} \frac{1}{n^2} \E \|\hat{\theta}-\theta^*\|^2 \leq & C \:\frac{\log n}{n^{3/4}}.
\end{align*}
A natural question is whether the worst case MSE over $\Theta(1,1)$ actually scales like $\tilde{O}(n^{-3/4})$. We do not know the answer to this question. However, we do have an explicit example of $\theta^*\in \Theta(1,1)$ for which the quantity $Q(1,\theta^*)$ actually scales like $n^{\frac{5}{4}}$, which suggests  
that it is not possible to prove a better rate than $\tilde{O}(n^{-3/4})$ with our proof technique.
\end{remark}

\textit{Even though the worst case risk over the class of Lipschitz matrices is $\tilde{O}(n^{-3/4})$, under an extra assumption that $\theta$ is lower Lipschitz as well, we get an improved MSE of $\tilde{O}(n^{-1})$}. More generally, the same holds for lower Holder continuous matrices as well. To make this precise we propose the following definition:
\begin{definition}
For $\alpha, L, L'>0$ let $\Theta(\alpha,L,L')\subset \Theta(\alpha,L)$ denote the subset of all lower Holder continuous matrices with lower Holder constant $L'$, i.e. after being permuted in rows and columns by some permutation it satisfies 
$$|\theta_{i,j}-\theta_{i,k}|\ge L'\frac{|j-k|^\alpha}{n^\alpha},$$
for all $i \neq j,i \neq k,i,j,k\in [n]$. 
\end{definition}

\begin{corollary}\label{cor:lipstrict}
There exists a universal constant $C<\infty$ such that for any $p\ge \frac{2\log n}{n}$ we have
$$\sup_{\theta^* \in \Theta(\alpha,L,L')} \frac{1}{n^2} \E \|\hat{\theta}-\theta^*\|^2\le C\Big[\Big(\frac{L}{L'}\Big)^2\frac{\log n}{np}+\frac{(\log n)^2}{np}\Big].$$
\end{corollary}

\begin{remark}
A similar estimator (consists of a sorting step and a smoothing step) as ours was proposed in~\cite{CA14} in the symmetric model, where the authors assume that the true $\theta$ is both lower and upper Lipschitz (but not necessarily monotonic). Under an extra assumption on the sparsity of the gradient of the histogram of $\theta$, \cite[Theorem 3]{CA14}  shows that their estimator has MSE scaling like $\widetilde{O}(n^{-1})$. We obtain the same error bound without the sparsity assumption on the gradient, but under the extra assumption of monotonicity. Another point worth mentioning here is that monotonicity constraints make it possible for our estimator to be completely tuning parameter free while the estimator proposed in~\cite{CA14} has a bandwidth parameter that needs to be tuned. 

\end{remark}

\subsection{Main Contribution}

Our main contribution in this problem is to obtain a single result (Theorem~\ref{main}) encapsulating our understanding of how the MSE varies with the underlying $\theta^*.$  
The upper bound in Theorem~\ref{main}, being a sum of two terms, has a natural interpretation of being the minimax rate plus an extra term arising because of potential mistakes made in the sorting step. Operationally, Theorem~\ref{main} shows that a recipe to obtain an upper bound of the MSE at a particular $\theta^*$ is to upper bound the term $Q(p,\theta^*)$ which is a completely deterministic term. This recipe then furnishes several corollaries for submatrices of interest with the proofs of the corollaries now being very simple. Theorem~\ref{main} therefore shows adaptive rates are possible in our problem, even in the missing observations at random setting, when the true underlying matrix $\theta^*$ has additional structure. As of now, such adaptivity is not known to hold for the other competing estimator in this problem; the USVT estimator, proposed in~\cite{ChaUSVT15}.

At the later stages of preparing this document, we became aware of an independent work by~\citet{shah2016feeling} who analyze a similar estimator as ours. A couple of comments are in order to relate the results obtained here to those in~\cite{shah2016feeling}. Our estimator is not exactly the same as the CRL estimator proposed in~\cite{shah2016feeling} because they have an extra randomization step in constructing the sorting permutation. This helps in obtaining a MSE scaling like $\tilde{O}(1/n^{3/2})$ instead of $\tilde{O}(1/n)$ for the very special case when $\theta$ is constant (or nearly constant). In terms of worst case risk the performance of both estimators is the same, and equals $\tilde{O}(1/\sqrt{n})$. However, for many sub parameter spaces of interest our results provides sharper results, demonstrating the adaptive nature of our estimate. For instance, when $\theta^* \in \Theta^{(2)}$ with equal sized blocks, Theorem $2$ in~\cite{shah2016feeling} give a $\tilde{O}(1/\sqrt{n})$ upper bound while Theorem~\ref{main} attains the $\tilde{O}(1/n)$ which is minimax rate optimal for $\Theta^{(2)}$ upto log factors. The adaptation results for matrices of the Generalized Bradley Terry form or for matrices satisfying Holder smoothness conditions also do not follow from Theorem $2$ in~\cite{shah2016feeling}. Moreover, all our results are in the missing observations at random setting while \cite{shah2016feeling} work in the complete observations setting.
\subsection{Scope of future Work}\label{discuss}
An important open question is whether there exists a polynomial time estimator achieving the minimax rate of estimation of $\tilde{O}(n^{-1})$ (when there is no missing data). 
Another natural question is whether the upper bounds for our estimator given in Corollary~\ref{uppersmooth} for Lipschitz matrices and other smooth matrices are tight.  Finally, automatic adaptation properties of the Singular Value Threshholding estimator (proposed in~\cite{ChaUSVT15} and studied in~\cite{shah2016stochastically}) also deserve attention. 



\section{Proof of corollaries}\label{sec:cor}

\begin{proof}[Proof of Corollary \ref{cor:worst}]
To begin, fix $i,j\in [n]$ such that $|R_i(p\theta^*)-R_j(p\theta^*)|\le 4\sqrt{np\log n}$. Then we have
\begin{align*}
\sum_{k\in [n]}^n[p\theta^*_{ik}-p\theta^*_{jk}]^2
\le p\sum_{k\in [n]}^n|p\theta^*_{ik}-\theta^*_{jk}|
=p|R_i(p\theta^*)-R_j(p\theta^*)|
\le 4p\sqrt{np\log n}.
\end{align*}
The above bound implies $Q(p,\theta^*)\le 4(np)^{3/2}\sqrt{\log n}$, which along with Theorem~\ref{main} completes the proof of the corollary.
\end{proof}

\begin{proof}[Proof of Corollary \ref{cor:block}]
As before, it suffices to control the terms in $Q(p,\theta^*)$. To this end let $B$ be the underlying $k\times k$ matrix, i.e. $$\theta^*_{i,j}=B\Big(\Big\lceil \frac{ki}{n}\Big\rceil, \Big\lceil \frac{kj}{n}\Big\rceil\Big).$$ Now fix $i,j\in [n]$ such that $|R_i(p\:\theta^*)-R_j(p\:\theta^*)|\le 4\sqrt{np\log n}$. Then with $s:=\Big\lceil \frac{ki}{n}\Big\rceil$ and  $t:=\Big\lceil \frac{kj}{n}\Big\rceil$ we have
$$4\sqrt{np\log n}\ge |R_i(p\:\theta^*)-R_j(p\:\theta^*)|=\frac{np}{k}\Big|\sum_{r\in [k]}(B(s,r)-B(t,r))\Big|,$$
which in particular means $$\max_{r\in [k]}|B(s,r)-B(t,r)|\le \min\Big(1,4k\sqrt{\frac{\log n}{np}}\Big).$$ This gives
\begin{align*}
\|p\:\theta^*_{i.}-p\theta^*_{j.}\|^2
=&\frac{np^2}{k}\sum_{r\in [k]}(B(s,r)-B(t,r))^2\\
\leq\: &p\max_{r\in [k]}|B(s,r)-B(t,r)|\times\frac{np}{k}\Big|\sum_{r\in [k]}(B(s,r)-B(t,r))\Big|\\
\leq\: &p\min\Big(1,4k\sqrt{\frac{\log n}{np}}\Big) \times 4\sqrt{np\log n}\\
=&p\min\Big(4\sqrt{np\log n},16k\log n\Big)\le 16p\min(k,\sqrt{np})\log n.
\end{align*}
The required bound then follows from Theorem \ref{main}.
\end{proof}

\begin{proof}[Proof of corollary \ref{cor:gbm}]
To begin, note that $$L'|w_i-w_j|\le |\theta^*_{i,k}-\theta^*_{j,k}|\le L|w_i-w_j|.$$
Now fixing $i,j\in [n]$ such that $|R_i(p\:\theta^*)-R_j(p\:\theta^*)|\le 4\sqrt{np\log n}$, we have
\begin{align*}
4\sqrt{np\log n}\ge p\sum_{k\in [n]}|\theta^*_{i,k}-\theta^*_{j,k}|\ge npL'|w_i-w_j|,
\end{align*}
and so
\begin{align*}
\sum_{k\in [n]}[p\:\theta^*_{i,k}-p\:\theta^*_{j,k}]^2\le np^2 L^2(w_i-w_j)^2\le 16p\Big(\frac{L}{L'}\Big)^2\log n,
\end{align*}
which along with Theorem \ref{main} completes the proof.
\end{proof}
 
 \begin{proof}[Proof of Corollary \ref{uppersmooth}]
As before it suffices to control the terms in $Q(p,\theta^*)$. To this end, fix $i,j\in [n]$ such that
$|R_i(p\:\theta^*)-R_j(p\:\theta^*)|\le 4\sqrt{np\log n}$. We claim that
\begin{align}\label{eq:holder}
\max_{k\in [n]}|\theta^*_{i,k}-\theta^*_{j,k}|\le \Big(4\sqrt{\frac{\log n}{np}}\Big)^{\frac{\alpha}{\alpha+1}} (4L)^{\frac{1}{\alpha+1}}.
\end{align}
We first complete the proof of the corollary, deferring the proof of \eqref{eq:holder}. To this end we have
\begin{align*}
\sum_{k\in [n]}[p \: \theta^*_{i,k}-p\:\theta^*_{j,k}]^2\leq\: &p \: \max_{k\in [n]}|\theta^*_{i,k}-\theta^*_{j,k}|\times \sum_{k\in [n]}|p\:\theta^*_{i,k}-p\:\theta^*_{j,k}|\\
\leq\: & p \Big(4\sqrt{\frac{\log n}{np}}\Big)^{\frac{\alpha}{\alpha+1}} (4L)^{\frac{1}{\alpha+1}}\times 4\sqrt{np\log n}\\
=\:&16p\:\:L^{\frac{1}{\alpha+1}}\:(np)^{1-\frac{2\alpha+1}{2\alpha+2}}(\log n)^\frac{2\alpha+1}{2\alpha+2}\\
\leq\: &16p\:L^{\frac{1}{\alpha+1}}\:(np)^{1-\frac{2\alpha+1}{2\alpha+2}}\:\log n,
\end{align*}
from which the result follows on using Theorem \ref{main}.

It thus remains to complete the proof of \eqref{eq:holder}. To this end, let $\delta:=\max_{k\in [n]}|\theta^*_{i,k}-\theta^*_{j,k}|$, and let $k^\star\in [n]$ be such that $|\theta^*_{i,k^\star}-\theta^*_{j,k^\star}|=\delta$. Then setting $N:=\Big(\frac{\delta}{4L}\Big)^{1/\alpha}n$, for any $k\in [n]$ such that $|k-k^\star|\le N$ an application of triangle inequality gives
\begin{align*}
|\theta^*_{i,k}-\theta^*_{j,k}|\ge &|\theta^*_{i,k^\star}-\theta^*_{j,k^\star}|-|\theta^*_{i,k}-\theta^*_{i,k^\star}|-|\theta^*_{j,k}-\theta^*_{j,k^\star}|\\
\ge &\delta-\frac{2LN^\alpha}{n^\alpha}
=\frac{\delta}{2}.
\end{align*}
This in turn implies
\begin{align*}
4\sqrt{np\log n}\ge |R_i(p\:\theta^*)-R_j(p\:\theta^*)|\ge p \sum_{k:|k-k^\star|\le N}|\theta^*_{i,k}-\theta^*_{j,k}|\ge pN\delta=p\:\Big(\frac{\delta}{4L}\Big)^{1/\alpha}n\delta,
\end{align*}
which is same as $$\delta\le \Big(4\sqrt{\frac{\log n}{np}}\Big)^{\frac{\alpha}{\alpha+1}} (4L)^{\frac{1}{\alpha+1}}.$$
\end{proof}

 \begin{proof}[Proof of Corollary \ref{cor:lipstrict}]
As usual, fix $i,j\in [n]$ such that $|R_i(p\:\theta^*)-R_j(p\:\theta^*)|\le 4\sqrt{np\log n}$. Then we have
\begin{align*}
4\sqrt{np \log n}\ge |R_i(p\:\theta^*)-R_j(p\:\theta^*)|\ge np L'\Big(\frac{'|i-j|}{n}\Big)^\alpha,
\end{align*}
which implies
\begin{align*}
\sum_{k\in [n]}[p\:\theta^*_{i,k}-p\:\theta^*_{j,k}]^2\le np^2L^2\Big(\frac{|i-j|}{n}\Big)^{2\alpha}\le 16p\Big(\frac{L}{L'}\Big)^2\log n,
\end{align*}
from which the result follows using Theorem \ref{main}.
\end{proof}

\section{Proof of Theorem~\ref{main}}


To begin, note that the MSE of our estimator does not depend on the underlying true permutation $\pi^*$, and so 
henceforth we assume that the true underlying latent permutation $\pi^*$ is the identity permutation.
We now outline the proof of Theorem \ref{main} in several steps.

\subsection{Step 1}
Recall that our estimator is of the form $$\hat{\theta} = Proj\Big(\big(\frac{y - \frac{1}{2} J}{\hat{p}} + \frac{1}{2} J\big) \circ \hat{\sigma}\Big) \circ \hat{\sigma}^{-1}.$$ on the set $n\hat{p}\ge 1$. We now show that it is sufficient for our purposes to analyse the estimator where $\hat{p}$ is replaced by $p.$ We now state a lemma making this precise. 
\begin{lemma}\label{step0}
Suppose $$\tilde{\theta}:= Proj\Big(\big(\frac{y - \frac{1}{2} J}{p} + \frac{1}{2} J\big) \circ \hat{\sigma}\Big) \circ \hat{\sigma}^{-1}$$ be the estimator constructed assuming $p$ known. If $p \geq \frac{1}{n}$, then there exists a constant $C < \infty$ such that for all $n \geq 2$ we have
$$ \frac{1}{n^2}  \E \|\hat{\theta}-\tilde{\theta}\|^2 \leq \frac{C}{n^2 p^2}.$$
\end{lemma}

The proof of Lemma~\ref{step0} is given in the appendix. Lemma~\ref{step0} therefore lets us conclude 
\begin{equation*}
\frac{1}{n^2} \E \|\hat{\theta} - \theta^*\|^2 \leq \frac{C}{n^2}\Big[ \E \|\tilde{\theta} - \theta^*\|^2 + \frac{1}{ p^2}\Big].
\end{equation*}

Hence from now on, we will analyze the mean squared error of the estimator $\tilde{\theta}.$  

\subsection{Step 2} 
We now write the following inequality.
\begin{align*}
\|\tilde{\theta} - \theta^*\| =& \|Proj\Big(\big(\frac{y - \frac{1}{2} J}{p} + \frac{1}{2} J\big) \circ \hat{\sigma}\Big) \circ \hat{\sigma}^{-1} - \theta^* \circ \hat{\sigma}^{-1} + \theta^* \circ \hat{\sigma}^{-1} - \theta^*\| \\
\leq &\|Proj\Big(\big(\frac{y - \frac{1}{2} J}{p} + \frac{1}{2} J\big) \circ \hat{\sigma}\Big) \circ \hat{\sigma}^{-1} - \theta^* \circ \hat{\sigma}^{-1}\| + \|\theta^* \circ \hat{\sigma}^{-1} - \theta^*\|  \\
=& \|Proj\Big(\big(\frac{y - \frac{1}{2} J}{p} + \frac{1}{2} J\big) \circ \hat{\sigma}\Big) - \theta^*\| + \|\theta^* \circ \hat{\sigma}^{-1} - \theta^*\|. 
\end{align*}
In this step we handle the first term above. 
Fix $\theta^* \in \T \subset [0,1]^{n \times n}$ and define a random (depends on $y$) function $f_{\theta^*}: \R_{+} \rightarrow \R$ by
\begin{equation*}
f_{\theta^*}(t) = \sup_{\theta \in \T: \|\theta -\theta^*\| \leq t} \langle \big(\frac{y - \frac{1}{2} J}{p} + \frac{1}{2} J\big) \circ \hat{\sigma} \:\:-\:\: \theta^*, \theta - \theta^* \rangle
\end{equation*}
Here the $\langle A,B \rangle$ notation refers to the inner product of two matrices defined as $Trace(A^{T} B).$ The following proposition connects the term $\|Proj\Big(\big(\frac{y - \frac{1}{2} J}{p} + \frac{1}{2} J\big) \circ \hat{\sigma}\Big) - \theta^*\|$ to the function $f_{\theta^*}.$


\begin{proposition}\label{leminterm}
The function $f_{\theta^*}(t)$ defined  above is strictly concave with $f(0) = 0$ and $f_{\theta^*}$ converges to $-\infty$ as $t\rightarrow\infty$.  Denoting the unique maximizer of the function $f_{\theta^*}$ by $t^*,$ we have
$\|Proj\Big(\big(\frac{y - \frac{1}{2} J}{p} + \frac{1}{2} J\big) \circ \hat{\sigma}\Big) - \theta^*\|= t^*$.
Moreover, if $s > 0$ satisfies $f_{\theta^*}(s) \leq 0$ then
$t^* \leq s.$

\end{proposition}

Proposition~\ref{leminterm} is proved using a representation result for the projection of a vector onto a convex set as developed in~\cite{Chat14}. Note that this is a deterministic result and does not depend on the distributional properties of $y.$ For the sake of completeness we prove this result in the appendix. 
For a more general version of the above proposition see Lemma 3.1 in~\cite{chatterjee2015adaptive}, where the projection is onto sets which are a finite union of convex sets.
%



\subsection{Step 3}
Proposition~\ref{leminterm} reduces the problem of upper bounding $\|Proj\Big(\big(\frac{y - \frac{1}{2} J}{p} + \frac{1}{2} J\big) \circ \hat{\sigma}\Big) - \theta^*\|$ to the problem of finding a $s > 0$ such that $f_{\theta^*}(s) \leq \frac{s^2}{2}$ which in turn behooves us to find a good upper bound of $f_{\theta^*}(t)$ for any fixed $t.$ Proceeding to do this,
define the mean zero random variables $v_{ij} := y_{ij} - 1/2 - p (\theta^*_{ij} - 1/2),$ and note that  
\begin{equation*}
\big(\frac{y - \frac{1}{2} J}{p} + \frac{1}{2} J\big) \circ \hat{\sigma} = \theta^* \circ \hat{\sigma} + \frac{1}{p} v \circ \hat{\sigma}.
\end{equation*}
where $v$ is a matrix with $v_{ij}$ on the non diagonals and zero on the diagonals. This implies 
\begin{equation*}
f_{\theta^*}(t) \leq \sup_{\theta \in \T: \|\theta - \theta^*\| \leq t} \langle [\theta^* \circ \hat{\sigma}] -  \theta^*, \theta - \theta^* \rangle  + \sup_{\theta \in \T: \|\theta - \theta^*\| \leq t} \langle \frac{1}{p} v \circ \hat{\sigma}, \theta - \theta^* \rangle. 
\end{equation*}
An application of the Cauchy Schwarz Inequality to the first term on the right side of the above inequality 
now gets us the following:
\begin{equation}\label{mark1}
f_{\theta^*}(t) \leq t \|[\theta^* \circ \hat{\sigma}] -  \theta^*\|  +\frac{1}{p} M_t,\quad M_t:=  \sup_{\pi\in S_n}\sup_{\theta \in \T: \|\theta - \theta^*\| \leq t} \langle v \circ \pi, \theta - \theta^* \rangle.
\end{equation}
%
%
%
The control  on $M_{t}$ is carried out in the following lemma. The proof is done by a chaining argument and is given in the appendix. 

\begin{lemma}\label{gausup}
There exists universal positive constants $C,c$ such that for any $n \ge 2, \theta^* \in \T \subset [0,1]^{n \times n}$ we have
\begin{equation*}
\P(M_t > C (1+t)n (\log n)^2)\le e^{-c n (\log n)^2}.
\end{equation*}
\end{lemma}

Combining Lemma \ref{gausup}  with~\eqref{mark1} shows that with high probability as quantified in Lemma \ref{gausup} we have
\begin{equation*}
f_{\theta^*}(t) \leq t \|[\theta^* \circ \hat{\sigma}] -  \theta^*\|  + (1 + t) C \frac{n}{p} (\log n)^2.
\end{equation*}
Setting $s^2 = \max\{\|\theta^* \circ \hat{\sigma} - \theta^*\|^2, C \frac{n}{p} (\log n)^2\}$ it can be checked that $f_{\theta^*}(s) \leq s^2/2.$ An application of Proposition~\ref{leminterm} then gives us $\|\tilde{\theta} - \theta^*\|^2 \leq s^2$ with high probability, which in turn gives the following 
proposition: 
\begin{proposition}\label{key}
There exists a positive constant $C$ such that for all $n\ge 2$ and $p\ge \frac{2}{n}$ we have
\begin{equation*}
\E\frac{1}{n^2} \|\tilde{\theta} - \theta^*\|^2 \leq C\Big[ \frac{(\log n)^2}{np} + \frac{1}{n^2}\:\E \|(\theta^* \circ \hat{\sigma}) - \theta^*\|^2\Big].
\end{equation*}
\end{proposition} 
The above proposition upper bounds the risk of our estimator by a sum of two terms. The first term scaling like $(\log n)^2/np$ term is essentially the minimax rate of our problem (upto logarithmic factors) which is attained by the global least squares estimate (c.f. \cite[Theorem 5]{shah2016stochastically}). The second term is the excess risk of our estimator as compared to the risk of the global least squares estimate. Thus it suffices to focus on controlling this excess risk term, 
which measures how much the sorting permutation $\hat{\sigma}$ changes $\theta^*$.

\subsection{Step 4}
In this step we investigate the excess risk term $\frac{1}{(np)^2 } \|p\:\theta^* \circ \hat{\sigma} - p\:\theta^*\|^2.$ Analyzing this term is one of the key contributions of this paper. We first prove the following lemma concerning the behavior of $\hat{\sigma}$. Recall that for any $\theta \in \Theta$ we denote $R_i(\theta):=\sum_{j=1}^n\theta_{ij}$ to be the $i$th row sum of the matrix $\theta.$
\begin{lemma}\label{thmperm}
Let $p > 2 \log n/n.$ Setting $t_n:=2\sqrt{np\log n}$ we have 
\begin{align}\label{eq:thmperm}
&\P\Big(\max_{i\in [n]}|R_i(p\:\theta^*)-R_{\hat{\sigma}(i)}(p\:\theta^*)|>2t_n\Big)\le 2n^{-1},
\end{align}
\end{lemma}

\begin{proof}
Define the vector of row sums of the data matrix $y$ as $r = (r_1,\dots,r_n)$ where $r_i = \sum_{j = 1}^{n} y _{ij}$ for all $1 \leq i \leq n.$ Setting 
\begin{equation}\label{event}
A_n:=  \Big\{\max_{i\in [n]}\Big|r_i-R_i(p\:\theta^*)-\frac{(n-1)(1 - p)}{2}\Big|\le t_n\Big\}
\end{equation} 
we will first show that 
\begin{align}\label{eq:prob_bound}
\P(A_n^c)\le 2n^{-1}.
\end{align} With $$\sigma_i^2:=\sum_{j=1}^n Var(y_{ij})=\frac{(n-1)p(1-p)}{4}+p^2\sum_{j=1}^n\theta_{ij}(1-\theta_{ij})\le \frac{np}{2},$$ 
we have
\begin{align*}
\frac{t_n^2}{2\sigma_i^2}\ge 4\log n,\quad 
\frac{3t_n^2}{2t_n}=\frac{3t_n}{2}\ge 3\sqrt{np\log n}\ge 4\log n.
\end{align*}
Thus an application of Bernstein's inequality gives
$$\P(|r_i-R_i(p\:\theta^*)-\frac{(n-1)(1 - p)}{2}|> t_n)\le 2\text{exp}\Big\{-\frac{t_n^2/2}{\sigma_i^2+t_n/3}\Big\}\le 2e^{-2\log n}=2n^{-2}.$$
A union bound then proves \eqref{eq:prob_bound}.

It thus suffices to show that after conditioning on the event $A_n$ we have for all $1 \leq i \leq n,$
\begin{align*}\label{eq:argument_needed}
|R_i(p\:\theta^*)-R_{\hat{\sigma}(i)}(p\:\theta^*)|\le  2t_n.
\end{align*}
 For this, setting $\hat{\sigma}(i)=j$, we split the proof into the following cases:
\begin{itemize}
\item{ $\hat{\sigma}(i)<i$}

In this case we will show that

$$|R_i(p\:\theta^*)-R_{\hat{\sigma}(i)}(p\:\theta^*)| = R_i(p\:\theta^*)-R_{\hat{\sigma}(i)}(p\:\theta^*) \le 2t_n.$$
If not, then we have $R_i(p\:\theta^*)+\frac{(n-1)(1 - p)}{2}-t_n > R_j(p\:\theta^*)+\frac{(n-1)(1 - p)}{2}+t_n$, and so the intervals  $[R_i(p\:\theta^*)+\frac{(n-1)(1 - p)}{2}- t_n,R_i(p\:\theta^*)+\frac{(n-1)(1 - p)}{2}+ t_n]$ and $[R_j(p\:\theta^*)+\frac{(n-1)(1 - p)}{2}- t_n,R_j(p\:\theta^*)+\frac{(n-1)(1 - p)}{2}+ t_n]$ are disjoint, implying $r_j < r_i$ and thus giving $i=\hat{\sigma}^{-1}(j)<\hat{\sigma}^{-1}(i)$. Now, for any $k>i$ we have $R_k(p\:\theta^*)-t_n\ge R_i(p\:\theta^*)-t_n$ by monotonicity.
 This implies that the intervals $[R_k(p\:\theta^*)+\frac{(n-1)(1 - p)}{2}- t_n,R_k(p\:\theta^*)+\frac{(n-1)(1 - p)}{2}+ t_n]$ and $[R_j(p\:\theta^*)+\frac{(n-1)(1 - p)}{2}- t_n,R_j(p\:\theta^*)+\frac{(n-1)(1 - p)}{2}+ t_n]$ are disjoint, and so
 $r_k> r_j$ which gives $\hat{\sigma}^{-1}(k)>\hat{\sigma}^{-1}(j)=i$. Since this holds for every $k>i$, the permutation $\hat{\sigma}^{-1}$ maps the set $\{i,i+1,\cdots,n\}$ to a subset of $\{i+1,\cdots,n\}$, which is impossible.

%

\item{$ \hat{\sigma}(i)\ge i$}

In this case  we will again show that
$$|R_i(p\:\theta^*)-R_{\hat{\sigma}(i)}(p\:\theta^*)| = R_{\hat{\sigma}(i)}(p\:\theta^*) - R_i(p\:\theta^*) \le 2t_n.$$
If this does not hold, with $j:=\hat{\sigma}(i)$ we have $R_j(p\:\theta^*)+\frac{(n-1)(1 - p)}{2}-t_n>R_i(p\:\theta^*)+\frac{(n-1)(1 - p)}{2}+t_n$. Consequently the intervals $[R_i(p\:\theta^*)+\frac{(n-1)(1 - p)}{2}- t_n,R_i(p\:\theta^*)+\frac{(n-1)(1 - p)}{2}+ t_n]$ and $[R_j(p\:\theta^*)+\frac{(n-1)(1 - p)}{2}- t_n,R_j(p\:\theta^*)+\frac{(n-1)(1 - p)}{2}+ t_n]$ are disjoint, and so $r_j>r_i$. By construction of $\hat{\sigma}$ we have $i=\hat{\sigma}^{-1}(j)>\hat{\sigma}^{-1}(i)$. Finally for any $k<i$ we have $$R_k(p\:\theta^*)+\frac{(n-1)(1 - p)}{2}+t_n\le R_i(p\:\theta^*)+t_n\le R_j(p\:\theta^*)+\frac{(n-1)(1 - p)}{2}-t_n,$$ and so the intervals $[R_k(p\:\theta^*)+\frac{(n-1)(1 - p)}{2}-t_n,R_k(p\:\theta^*)+\frac{(n-1)(1 - p)}{2}+t_n]$ and $[R_j(p\:\theta^*)+\frac{(n-1)(1 - p)}{2}-t_n,R_j(p\:\theta^*)+\frac{(n-1)(1 - p)}{2}+t_n]$ are disjoint as well. This  gives $r_k< r_j$, and consequently we have $\hat{\sigma}^{-1}(k)<\sigma^{-1}(j)=i$. Thus the permutation $\hat{\sigma}^{-1}$ maps the set $\{1,2,\cdots,i\}$ to a subset of $\{1,2,\cdots,i-1\}$, a contradiction.

\end{itemize}
\end{proof}

\begin{remark}
The ideal case here is when the permutation $\hat{\sigma}$ is close to the identity permutation. In general though, $\hat{\sigma}$ need not be close to the identity permutation. For example, when $\theta^*$ is a constant matrix, $\hat{\sigma}$ is close to a uniformly random permutation. But Lemma~\ref{key} shows that $R_{\hat{\sigma}(i)}$ and $R_i$ are always close irrespective of what $\theta^*$ is. For instance, when $p = 1$ we have $|R_{\hat{\sigma}} - R_i| = \tilde{O}(\sqrt{n})$  even though both $R_{\hat{\sigma}}$ and $R_i$ could be $O(n).$ 

If however the row sums $R_i(\theta^*)$ are strictly increasing, one immediately gets concentration of $\hat{\sigma}$ towards identity. In particular when $p = 1,$ if $\min_{i\in [n]}|R_i(\theta^*)-R_{i-1}(\theta^*)|$ is uniformly bounded away from $0$, then  Lemma \ref{thmperm} shows that high probability 
$$\max_{i\in [n]}|i-{\hat{\sigma}(i)}|=O(\sqrt{n\log n}).$$
Of course, if the row sums are not increasing, no concentration of $\hat{\sigma}$ towards identity is expected. 
\end{remark}

\begin{remark}
We also note that both the bounds above are adaptive in terms of sparsity of the underlying graph. For e.g. if the entries of the matrix $\theta$ are mostly $0$ or small, the row sums $R_i(\theta)$ will be small as well, thus giving a better bound. 
\end{remark}


Now we can now control the excess risk term $\frac{1}{n^2 p^2} \|p \theta^* \circ \hat{\sigma} - p \theta^*\|^2.$ This is done as follows.

%
Using Lemma~\ref{thmperm} we have
$$\P(B_n^{c})\le 2n^{-1},\quad B_n:=  \{\max_{i\in [n]}|R_i(p\:\theta^*)-R_{\hat{\sigma}(i)}(p\:\theta^*)|\le 4\sqrt{np\log n}\}.$$
This gives
\begin{align*}&\E \sum_{i,j=1}^n[p\:\theta^*_{ij}-p\:\theta^*_{\hat{\sigma}(i),\hat{\sigma}(j)}]^2\\
=&\E \sum_{i,j=1}^n[p\:\theta^*_{ij}-p\:\theta^*_{\hat{\sigma}(i),\hat{\sigma}(j)}]^21_{B_n^c}+\E \sum_{i,j=1}^n[p\:\theta^*_{ij}-p\:\theta^*_{\hat{\sigma}(i),\hat{\sigma}(j)}]^21_{B_n}\\
\le &n^2p^2\P(B_n^c)+\E\sum_{i,j=1}^n[p\:\theta^*_{ij}-p\:\theta^*_{\hat{\sigma}(i),\hat{\sigma}(j)}]^21_{B_n}\\
\le &2np^2+2\E \sum_{i,j=1}^n[p\:\theta^*_{ij}-p\:\theta^*_{\hat{\sigma}(i),j}]^21_{B_n}+2\E\sum_{i,j=1}^n[p\:\theta^*_{\hat{\sigma}(i),j}-p\:\theta^*_{\hat{\sigma}(i),\hat{\sigma}(j)}]^21_{B_n}\\
\le &2np^2+4Q(p,\theta^*).
\end{align*}
The last display along with Proposition~\ref{key} completes the proof of Theorem~\ref{main}.

\subsection{Proof of Theorem~\ref{minimax1}}\label{Minimax}

We need to use the following version of Gilbert Varshamov coding lemma (\cite{varshamov1957estimate}). The proof of this lemma is provided in subsection~\ref{gvproof}.
\begin{lemma}\label{gv}[Gilbert-Varshamov]
Fix any positive integer $d.$ Let $$\H_{1/2} = \{v \in \{0,1\}^d: \sum_{i = 1}^{d} v_i = d/2\}.$$ There exists a subset $\F \subset \H_{1/2}$ with 
\begin{equation}\label{cardinality}
|\F| \geq \exp(d/32)
\end{equation}
such that for any $w \neq w^{'} \in W$ we have 
\begin{equation}\label{hamming}
\frac{d}{2} \geq H(w,w^{'}) \geq \frac{d}{8}
\end{equation}
where $H$ refers to the Hamming distance between any two points of the hypercube.
\end{lemma}

We are now ready to prove Theorem~\ref{minimax1}.
\begin{proof}[Proof of Theorem~\ref{minimax1}]
We are going to prove this lower bound for the symmetric model and a similar proof can be constructed for the skew symmetric model. For any subset $S \subset \{1,\dots,d\},$ define a matrix $\theta^{S} \subset \Theta^{(2)}$ as follows whenever $i \neq j,$
\begin{equation*}
\theta^{S}_{ij}=\begin{cases} 0.5 &\mbox{if } i \in S, j \in S \\ 
0.5 + \frac{c}{\sqrt{np}} & \mbox{if } i \in S, j \notin S \\ 
0.5 +\frac{c}{\sqrt{np}} & \mbox{if } i \notin S, j \in S \\ 
0.5 + 2\frac{c}{\sqrt{np}} & \mbox{if } i \notin S, j \notin S
\end{cases}.
\end{equation*}

Take two different subsets $S,S^{'}$ of $\{1,\dots,d\}.$ Denote by $S \Delta S^{'}$ to be the symmetric difference of the two sets. Note that whenever $i \in S \Delta S^{'}$ and $j$ does not belong to $S \Delta S^{'}$ or vice versa, we have
\begin{equation*}
(\theta^{S}_{ij} - \theta^{S^{'}}_{ij})^2 \geq \frac{c^2}{np}.
\end{equation*}
Therefore we have
\begin{equation}\label{sep1}
\sum_{i = 1}^{n} \sum_{j = 1}^{n} (\theta^{S}_{ij} - \theta^{S^{'}}_{ij})^2 \geq 2 \frac{c^2}{np} |S \Delta S^{'}| |n - S \Delta S^{'}|.
\end{equation}

Now we apply Lemma~\ref{gv} to extract a finite subset $\F \subset \H_{1/2}$ satisfying~\eqref{cardinality} and~\eqref{hamming}. Note that a vector in $\{0,1\}^d$ can be thought of as a subset of $\{1,2,\cdots,d\}$ by considering the indices which equal $1$. With this identification, consider the set of matrices 
\begin{equation*}
W = \{\theta \in \Theta: \theta = \theta^{S} \:\:\text{for some}\:\: S \in {\F}.\}
\end{equation*}
Hence the cardinality $|S \Delta S^{'}|$ is exactly the Hamming distance between the corresponding binary strings.

Hence we can now apply~\eqref{hamming} to~\eqref{sep1} to get the following bound for any $\theta \neq \theta^{'}$ belonging to $W$:
\begin{equation}\label{sep2}
\sum_{i = 1}^{n} \sum_{j = 1}^{n} (\theta_{ij} - \theta^{'}_{ij})^2 \geq \frac{c^2 n}{8p}. 
\end{equation}

This implies that $W$ is a packing set of radius $c \sqrt{n/8p}$ with cardinality atleast $\exp(n/32).$ Proceeding to bound Kulback Leibler divergences, let the distribution of $y$ and $y_{ij}$ under $\theta$ be denoted by $P_{\theta}$ and $P^{(ij)}_{\theta}$ respectively. Choosing $c$ small ensures that $\theta^{'}_{ij},\theta_{ij}\in [0.4,0.6]$, which in turn gives
\begin{equation}\label{bernkl}
D(P^{(ij)}_{\theta},P^{(ij)}_{\theta^{'}}) = p\:D(Bern(\theta_{ij}),Bern(\theta^{'}_{ij})) = p \:\int_{\theta^{'}_{ij}}^{\theta_{ij}} \frac{\theta_{ij} - x}{x(1 - x)} dx \leq  7  \|\theta_{ij} - \theta^{'}_{ij}\|^2.
\end{equation} 
The above inequality implies $\max_{\theta \neq \theta^{'}} D(P_{\theta}, P_{\theta^{'}}) \leq 7 \max_{\theta \neq \theta^{'}} \|\theta - \theta^{'}\|^2 \leq 28 \frac{n}{p} c^2$ because each entry of $\theta - \theta^{'}$ is bounded in magnitude by $2 c/\sqrt{np}.$
A standard application of Fano's lemma (see Chapter 13 in~\cite{duchi2016lecture}) with $W$ as our packing set and using~\eqref{sep2} and~\eqref{bernkl}, we obtain a minimax lower bound
\begin{equation*}
\inf_{\tilde{\theta}} \sup_{\theta \in \Theta^{(2)}} \frac{1}{n^2} \sum_{i = 1}^{n} \sum_{j = 1}^{n} \E (\tilde{\theta}_{ij} - \theta_{ij})^2 \geq \frac{c^2}{8 np} \big(1 - 32 \frac{28 c^2 n + \log 2}{n}\big).
\end{equation*}

Choosing $c$ appropriately small enough small enough finishes the proof of the theorem. 
\end{proof}

\section{Acknowledgements}
We want to thank Bodhisattva Sen for introducing us to this problem, and for his helpful comments and suggestions.

\bibliographystyle{plainnat}
\bibliography{perm_iso_ref_rev}

\section{Appendix}\label{appendix}

\subsection{Proof of Lemma~\ref{step0}}
\begin{proof}
To begin, note that 
$\|\hat{\theta} - \tilde{\theta}\|^2 \leq n^2$, which gives 
\begin{align}\label{eq:chernoff}
\E\|\hat{\theta}-\tilde{\theta}\|^2 \leq \:&n^2\:\P(\hat{p}\notin [p/2,2p])+\E \big[\Big(\frac{1}{\hat{p}}-\frac{1}{p}\Big)^2\sum_{i,j=1}^n\Big(y_{ij}-\frac{1}{2}\Big)^2 \I\Big\{ \frac{p}{2}\le \hat{p}\le 2p\Big\}\big]
\end{align}
where the last term in the RHS of \eqref{eq:chernoff} can be bounded by
\begin{align*}
\frac{4}{p^4}\E \big[(\hat{p}-p)^2\sum_{i,j=1}^n\Big(y_{ij}-\frac{1}{2}\Big)^2 \I\Big\{ \frac{p}{2}\le \hat{p}\leq 2p\Big\}\big]
=& \frac{n^2}{p^4}\E \big[\hat{p}\:(\hat{p}-p)^2 \I\Big\{ \frac{p}{2}\leq \hat{p} \leq 2p \Big\}\big]\\
\le &\frac{2n^2}{p^3} \E (\hat{p}-p)^2
\le \frac{8}{p^2}.
\end{align*}
For the first two terms in the RHS of \eqref{eq:chernoff}, using the multiplicative form of Chernoff's bound we have
$$\P\Big(\hat{p}< \frac{p}{2}\Big)\le e^{-\frac{n(n-1)p}{16}}\le e^{-\frac{n-1}{8}},\quad \P\Big(\hat{p}>2p\Big)\le e^{-\frac{n(n-1)p}{6}}\le e^{-\frac{n-1}{3}}.$$
Thus plugging in the estimates in \eqref{eq:chernoff} we get
$$\frac{1}{n^2}\E \|\hat{\theta}-\tilde{\theta}\|^2 \leq n^2(e^{-\frac{n-1}{8}}+e^{-\frac{n-1}{3}})+\frac{8}{n^2p^2},$$
from which the result follows.
\end{proof}

\subsection{Proof of Lemma~\ref{gausup}}


In order to prove the above lemma, we will need to make use of the following two standard results. We first set up some notations. For any set $A \subset \R^n$ define its covering number at radius $\epsilon > 0$ to be the minimum number of Euclidean balls of radius $\epsilon$ centred inside $A$ such that the union of the balls is a superset of $A.$ Denote this covering number by $N(A,\epsilon)$ where the Euclidean metric used is implicit.



Let us now define the space of matrices which are non decreasing in both rows and columns with entries between $0$ and $1$ as
\begin{equation*}\label{monomat}
\M_{[0,1]} = \{\theta \in [0,1]^{n \times n}: \theta_{ij} \leq \theta_{kl}\:\:\text{iff}\:\:i \leq k\:\:\text{and}\:\:j \leq l\}.
\end{equation*}

Estimates of the metric entropy of $\M_{[0,1]}$ are available in the literature. The next proposition shows that these metric entropy bounds of $\M_{[0,1]}$ can be used to derive a similar bound for the covering number of $\T.$
\begin{proposition}\label{covprop}
Fix a positive integer $n.$ We have the covering number inequality for any $\eps > 0,$
\begin{equation*}
\log N(\eps,\T,\|.\|) \leq C\:\Big(\frac{n}{\eps}\Big)^2\:\:\Big[\log \Big(\frac{n}{\eps}\Big)\Big]^2.
\end{equation*}
The same upper bound for the log covering number holds when $\T$ is replaced by $\G.$
\end{proposition}

\begin{proof}
Let us prove the proposition for the set $\T$ and the corresponding statement for $\G$ will follow similarly. Let us define a map $D: \T \rightarrow [0,1]^{n \times n}$ as follows:
\begin{align*}
D(\theta)_{i,j} = \min\{\theta_{i,i-1},\theta_{i + 1,i}\}\:\I\{i = j\} + \theta_{i,j}\:\I\{i \neq j\}
\end{align*}
Let us define another map $\phi: [0,1]^{n \times n} \rightarrow [0,1]^{n \times n}$ as
\begin{align*}
\phi(\theta)_{i,j} = \theta_{i,n - j + 1}
\end{align*}
Define the composition map $f = \phi \circ D.$ Recall that $\M_{[0,1]}$ is the space of monotone matrices with entries in $[0,1].$ It can now be checked that $f(\theta) \in \M_{[0,1]}$ for all $\theta \in \T$ and $f$ is a one to one mapping. We now make the observation that for any $\theta \neq \theta^{'}$ belonging to $\T$ we have
\begin{equation}\label{comp}
\|\theta - \theta^{'}\|^2 \leq \|f(\theta) - f(\theta^{'})\|^2.
\end{equation}


Now also note that $f$ is a continuous map and hence its image $f(\T)$ is a closed subset of $\M_{[0,1]}.$ Hence there exists a covering set $F$ of $f(\T) \subset \M_{[0,1]}$ at radius $\eps$ with cardinality atmost $N(\eps/2,\M_{[0,1]}).$ This is because the set of projections of the covering set $F$ at radius $\eps/2$ onto the closed set $f(\T)$ forms a covering set of $f(\T)$ at radius $\eps.$ Therefore it follows from~\eqref{comp} that the inverse image of $F$ under the mapping $f$ is a covering set of $\T$ at radius $\eps.$ Thus we can conclude,
\begin{equation*}
\log N(\eps,\T) \leq \log N(\eps/2,\M_{[0,1]}).
\end{equation*}
The covering number for $\M_{[0,1]}$ can be obtained using Lemma 3.4 in~\cite{chatterjee2015matrix} and is written below. This result is proved by using the covering number results in~\cite{GW07}.
\begin{equation*}
\log N(\eps,\M_{[0,1]}) \leq C\:\Big(\frac{n}{\eps}\Big)^2\:\:\Big[\log \Big(\frac{n}{\eps}\Big)\Big]^2
\end{equation*}
where $\eps > 0$ and $C$ is a universal constant. The last two displays finish the proof. The same proof for $\G$ goes through by defining $\phi$ to be the identity map.
\end{proof}

We are now ready to prove Lemma~\ref{gausup}.
\begin{proof}[Proof of Lemma~\ref{gausup}]
To begin, setting $M_{\pi,t}:=\sup_{\theta \in \T: \|\theta - \theta^*\| \leq t} \langle v \circ \pi, \theta - \theta^* \rangle$ we claim that there exists a universal constant $C$ such that 
\begin{equation}\label{expmpit}
\E M_{\pi,t} \leq C n (\log n)^2.
\end{equation}
We first prove the lemma, deferring the proof of \eqref{expmpit}. Note that by definition, each entry of $v$ is zero mean and lies between $-1$ and $1.$ 
Also in the skew symmetric model, for any $i < j$ 
we have $v_{ij}=-v_{ji}$, where we use the fact that we are filling the missing entries by $1/2.$ Hence the contribution of the upper diagonal part to $M_{\pi,t}$ is the same as the lower part, and so
\begin{equation}\label{innpro}
M_{\pi,t} = 2 \sup_{\theta \in \T: \|\theta - \theta^*\| \leq t} \sum_{1 \leq i < j \leq n} v_{ij} (\theta_{ij} - \theta^*_{ij}).
\end{equation}
It is now not too hard to show that $M_{\pi,t},$ for any fixed $t > 0,$ is a Lipschitz function of $v$ with Lipschitz constant $t.$ Invoking a concentration result for convex Lipschitz functions due to Michel Ledoux(see~\cite[Theorem 6.10]{boucheron2013concentration}), we have

\begin{equation*}
\P(M_{\pi,t} > \E M_{\pi,t} + t C n (\log n)^2)\le e^{-c'n(\log n)^2}.
\end{equation*}
This, along with Lemma~\ref{expmpit} shows that with high probability we have $M_{\pi,t}\le  (1 + t) C n (\log n)^2$, which by a simple union bound argument gives
\begin{equation*}
\P( \sup_{\pi \in S_n} M_{\pi,t} > (1 + t) C n (\log n)^2)\le  n! \exp(-c' n (\log n)^2)\le e^{-cn(\log n)^2},
\end{equation*}
and so the proof of the Lemma is complete.



It thus remains to prove \eqref{expmpit}. For this, define
\begin{equation*}
\A = \{(\theta - \theta^*) \circ \pi: \theta \in \T\}.
\end{equation*}
Recall that $v$ is a random matrix with independent zero mean entries in the upper diagonal and also each entry takes values in $[-1,1].$ Hence an application of a standard chaining result (see \cite{VandegeerBook}) for subgaussian random variables results in the upper bound 
\begin{equation*}\label{cov11}
\E M_{\pi,t} \leq 12 \inf_{0 < \delta \leq n} \left\{\int_{\delta}^{n}
      \sqrt{\log N(\eps,\A)}\:d\eps + 4 n \delta\right\}. 
\end{equation*}
The upper limit of the integral is $n$ since the diameter (maximum Euclidean pairwise distance) of $\A$ is bounded by $n.$ Setting $\delta = \frac{1}{n}$ in the last equation we then obtain 
\begin{equation*}\label{cov1}
\E M_{\pi,t} \leq 12 \left\{\int_{\frac{1}{n}}^{n} C\:\Big(\frac{n}{\eps}\Big)\:\:\Big[\log \Big(\frac{n}{\eps}\Big)\Big] \:d \eps \right\} + 4. 
\end{equation*}
Note that $\log \Big(\frac{n}{\eps}\Big) \leq \log n^2$ because the variable of integration $\eps$ ranges from $1/n$ to $t.$ This then implies the upper bound
\begin{equation*}\label{cov1}
\E M_{\pi,t} \leq C \frac{n}{p} (\log n)^2.
\end{equation*}


\end{proof}

\subsection{Proof of Lemma~\ref{gv}}\label{gvproof}
\begin{proof}
Assume $4$ divides $d$ for simplicity. Let $H(.,.)$ denote the Hamming distance on $\{0,1\}^d$, i.e. for any $v_1,v_2\in \{0,1\}^d$ we have
$$H(v_1,v_2):=\sum_{i=1}^n|v_1(i)-v_2(i)|.$$ Take $\beta=\frac{1}{2}$ and $\alpha=\frac{1}{8}$, and $\gamma=32$. With $d=4k$ and setting
 $\widetilde{P}:=\{w\in \{0,1\}^{2k}:\sum_{i=1}^{2k}w_i=k\}$ we claim that there exists a set $\widetilde{F}\subset\widetilde{P}$ such that
 $|\widetilde{F}|\ge e^{d/32}$, and for any $w_1,w_2\in \widetilde{F}$ we have $H(w_1,w_2)\ge \frac{d}{16}$ . Given the claim, define $F$ as 
 $$F:=\{v\in \{0,1\}^d:(v_1,\cdots,v_{2k})\in \widetilde{F}, v_{2k+1}=v_{2k+2}=\cdots=v_{3k}=1,v_{3k+1}=\cdots=v_{4k}=0\},$$
 and note that for any $v\in F$ we have $\sum_{i=1}^{4k} v_i=2k$. Also $$v_1,v_2\in \widetilde{P}\Rightarrow \frac{d}{16}\le H(v_1,v_2)\le 2k=\frac{d}{2}.$$
 Since $|F|=|\widetilde{F}|\ge e^{d/32}$,  the proof of the lemma is complete.
 \\

 It thus remains to prove the claim. To this effect, let $\widetilde{F}$ denote the largest packing set of $\widetilde{P}$ of radius $\frac{d}{16}$ in Hamming distance, i.e. for any $w_1,w_2\in \widetilde{F}$ we have $H(w_1,w_2)\ge \frac{d}{16}$. Then using the maximaility of $\widetilde{F}$ we have
 $$\widetilde{P}= \bigcup_{w\in \widetilde{F}}B_H\Big(w,\frac{d}{8}\Big)\Rightarrow |\widetilde{P}|\le |\widetilde{F}|\times \sup_{w\in \widetilde{P}}\Big|B_H\Big(w,\frac{d}{8}\Big)\Big|,$$
which implies
 \begin{align}
 \frac{1}{|\widetilde{F}|}\le \frac{\sup_{w\in \widetilde{P}}\Big|B_H\Big(w,\frac{d}{8}\Big)\Big|}{|\widetilde{P}|}
 \le  \frac{\sup_{w\in \{0,1\}^d}\Big|B_H\Big(w,\frac{d}{8}\Big)\Big|}{2^{2k}} \frac{2^{2k}}{|\widetilde{P}|}
 =&\frac{\sup_{w\in \widetilde{P}}\P\Big(H(w,{\bf X})\le \frac{d}{8}\Big)}{ \P(\sum_{i=1}^{2k}X_i=k)}\label{eq:prob1},
 \end{align}
 where ${\bf X}=(X_1,\cdots,X_{2k})$ is a $2k$ dimensional random vector with uniform distribution on $\{0,1\}^{2k}$. To bound the numerator in the RHS above, note that the vector $(|w_1-X_1|,|w_2-X_2|,\cdots|w_n-X_n|)$ is also distributed uniformly on $\{0,1\}^{2k}$, and so the RHS of \eqref{eq:prob1} equals $ \frac{\P\Big(B\le \frac{k}{2}\Big)}{ \P(B=k)}$,
 where $B\sim Bin(2k, .5)$. 
 Using Hoeffding's inequality as
 $\P\Big(B\le \frac{k}{2}\Big)\le e^{-2\frac{(k/2)^2}{2k}}=e^{-k/4},$  whereas Stirling's approximation gives
 $\P(B=k)=\frac{{2k\choose k}}{2^{2k}}\sim \frac{1}{\sqrt{\pi k}}.$
 Combining these two estimates and using \eqref{eq:prob1}, for all large $k$ we have
 $|\widetilde{F}|\ge e^{k/8}$, thus completing the proof of the lemma.
 
\end{proof}



\begin{proof}[Proof of Proposition \ref{leminterm}]
This theorem actually holds for a general convex set $C \subset \R^d$ and points $v \in \R^d, \theta^* \in C.$ The proof can be read by setting $C = \mathcal{T}, v = \big(\frac{y - J/2}{p} + J/2\big) \circ \hat{\sigma}, d = n \times n.$ We first prove strict concavity of $f_{\theta^*}.$ Let $g(t):\R_{+} \rightarrow \R$ be defined as 
$$g(t) = \sup_{\theta \in C: \|\theta - \theta^{*}\| \leq t} \langle v - \theta^*,\theta - \theta^{*} \rangle.$$ If $t_0:= \inf_{v \in C} \|\theta^* - v\|$ is positive, then we
define $g(t) = - \infty$ whenever $t < t_0$. It is easy to check that $g$ is concave, and hence the function $f_{\theta^*}$ strictly concave.
%
%
%
Also we have  $f_{\theta^*}(0) = 0,$ and an application of Cauchy Schwarz inequality gives $f_{\theta^*}(t) \leq t \|v - \theta^*\| - t^2/2$, and hence $f_{\theta^*}(t)$ converges to $-\infty$ as $t \rightarrow \infty.$ These facts then imply that $f_{\theta^*}(.)$ has a unique maximizer which we denote by $t^*.$ It thus remains to show $\|\Pi_{C} v - \theta^*\| = \argmax_{t > 0} f_{\theta^*}(t)$, where 
 $\Pi_{C}(v)$ denotes the unique projection of $v$ onto $C.$ 
 To this effect,  noting that
\begin{align*}
\Pi_{C} v 
=\argmax_{\theta \in C} \left(\langle v - \theta^*, \theta - \theta^* \rangle - \|\theta - \theta^*\|^2/2\right),
\end{align*}
we can write
\begin{align*}
\|\Pi_{C} v - \theta^*\| &= \argmax_{t > 0} \sup_{\theta \in C:
  \|\theta - \theta^*\| = t} \langle v - \theta^*, \theta - \theta^* \rangle -
\frac{t^2}{2} \nonumber = \argmax_{t > 0} h(t), \label{deter1}
\end{align*}
where $h: \R_{+} \rightarrow \R$ is defined by
$h(t) = \sup_{\theta \in C: \|\theta - \theta^*\| = t} \langle v - \theta^*, \theta - \theta^* \rangle - \frac{t^2}{2}.$
We will now show that $t^*$ is a
unique maximizer of $h$ as well, for which first note that $h(t) \leq f_{\theta^*}(t)$ for all $t\ge 0$.  Recall that
Let $\tilde{\theta} \in
\{\theta \in C: \|\theta - \theta^*\| \leq t^*\}$ be a point where
$f_{\theta^*}(t^*)$ is achieved. If $\|\tilde{\theta} - \theta^*\| = t_1 <
t^*$, then we would have $f_{\theta^*}(t_1) > f_{\theta^*}(t^*)$, contradicting the
definition of $t^*$. Hence $\|\tilde{\theta} - \theta^*\| =
t^*$, implying that $f_{\theta^*}(t^*) = h(t^*)$.
This shows that $t^*$ is a unique maximizer of $h$ as well, and so
\begin{align*}
\|\Pi_{C} v - \theta^*\| &= \argmax_{t > 0} h(t) = \argmax_{t > 0} f_{\theta^*}(t),
\end{align*}
thus finishing the proof of the lemma.
\end{proof}

\end{document}